\documentclass[12pt]{amsart}
\usepackage{amsmath,amscd}
\usepackage{amssymb}
\usepackage{amsxtra}
\usepackage{enumerate}
\usepackage[mathscr]{eucal}
\usepackage[margin=2.1 cm]{geometry}

\usepackage{hyperref}
\usepackage{amsmath}
\usepackage{amsthm}
\usepackage{amssymb}
\usepackage{tikz}
\usepackage{float}
\usepackage{mathrsfs}
\usepackage[hypcap=false]{caption}
\usepackage{caption}
\usepackage{pst-node}
\usepackage{tikz-cd}
\bibliography{references}
\newtheorem{theorem}{Theorem}[section]
\newtheorem{lemma}[theorem]{Lemma}
\newtheorem{proposition}[theorem]{Proposition}

\newtheorem{definition}[theorem]{Definition}
\theoremstyle{definition}

\newtheorem{remark}[theorem]{Remark}

\newcommand{\h}{{\mathbb{H}}}

\newcommand{\secref}[1]{Section~\ref{#1}}
\newcommand{\thmref}[1]{Theorem~\ref{#1}}
\newcommand{\lemref}[1]{Lemma~\ref{#1}}

\newcommand{\eqnref}[1]{Equation~\ref{#1}}

\usepackage{hyperref}

\usepackage{graphicx}
\usepackage{subfig}
\graphicspath{ {./images/} }

\numberwithin{equation}{section}

\begin{document}
\title[Fundamental domain on Complex Bidisk]{EQUIDISTANT HYPERSURFACES OF THE COMPLEX BIDISK $\mathbb{H}^2_{\mathbb{C}}\times \mathbb{H}^2_{\mathbb{C}}$
}

\author[K. Gongopadhyay]{KRISHNENDU GONGOPADHYAY}

\author[L. Kundu]{LOKENATH KUNDU}

\author[A. Tiwari]{ADITYA TIWARI}

\address{
Indian Institute of Science Education and Research (IISER) Mohali,
		Knowledge City, Sector 81, SAS Nagar, Punjab 140306, India}
	\email{krishnendu@iisermohali.ac.in}

\address{Indian Institute of Science Education and Research (IISER) Mohali,
		Knowledge City, Sector 81, SAS Nagar, Punjab 140306, India.}
	\email{lokenath@iisermohali.ac.in}
\address{
Indian Institute of Science Education and Research (IISER) Mohali,
		Knowledge City, Sector 81, SAS Nagar, Punjab 140306, India}
	\email{adityatiwari@iisermohali.ac.in}

\dedicatory{ In memory of Todd A. Drumm}
\makeatletter
\@namedef{subjclassname@2020}{\textup{2020} Mathematics Subject Classification}
\makeatother

\subjclass[2020]{}

\keywords{}

\begin{abstract}
We consider the isometries of the complex hyperbolic bidisk, that is, the product space \( \mathbb{H}^2_{\mathbb{C}} \times \mathbb{H}^2_{\mathbb{C}} \), where each factor \( \mathbb{H}^2_{\mathbb{C}} \) denotes the complex hyperbolic plane.  We investigate the Dirichlet domain formed by the action of a cyclic subgroup $(g_1, g_2)$, where each $g_i$ is loxodromic. We prove that such a Dirichlet domain has two sides. 
\end{abstract}

\maketitle

\section{Introduction}

Suppose a discrete group $G$ is acting properly discontinuously and co-compactly on a metric space  $X$. The Dirichlet domain associated to this action is the convex subset defined as 
\begin{equation*}
    D(x_0)= \{ x\in X~|~ d(x,x_0)\leq d(g(x),x_0) ~\forall g \in G \setminus\{1\}\}
\end{equation*} for a given base point $x_0$. The Dirichlet domain gives a geometric way to understand the quotient space $X/G$. The Dirichlet domain is a fundamental polyhedron that tessellates $X$ without overlapping interiors. A well-known construction of a Dirichlet domain arises when 
$G$ is a discrete group of isometries acting on the hyperbolic plane.

One natural question is: how many sides can a Dirichlet domain for a cyclic group possess? This question is inspired by related investigations in other settings, e.g. J\o{}rgensen’s work on hyperbolic 3-space \cite{jor}, the work of  Drumm and Poritz \cite{port} and Phillips’s work on the fundamental domains for cyclic isometry groups of the complex hyperbolic space \cite{Philips}.

The bidisk, defined as the Cartesian product of two hyperbolic planes, frequently appears in the literature, particularly as a standard example in the study of symmetric spaces. However, relatively few papers specifically focus on its geometric properties, see e.g \cite{farb}, \cite{CDL}, \cite{port}. 
In \cite{CDL}, Drumm et. al. investigated the above problem for the bidisk 
${\h}^2 \times {\h}^2$. Drumm et. al. proved that a Dirichlet domain with a basepoint lying in the invariant flat must have exactly two faces. However, if the basepoint is chosen outside the flat, the Dirichlet domain may have more than two faces. 

Inspired by the work of Drumm and collaborators, we investigate the analogous question in the setting of the \emph{complex hyperbolic bidisk}, that is, the product space \( \mathbb{H}^2_{\mathbb{C}} \times \mathbb{H}^2_{\mathbb{C}} \), where each factor \( \mathbb{H}^2_{\mathbb{C}} \) denotes the complex hyperbolic plane equipped with the standard Bergman metric. We refer to this space simply as the \emph{complex bidisk}. It inherits a natural product metric $\rho$, defined as the square root of the Euclidean sum of the complex hyperbolic metrics evaluated component wise: 
\[
\rho\big((z_1, z_2), (w_1, w_2)\big) = \sqrt{d_{\mathbb{H}^2_{\mathbb{C}}}(z_1, w_1)^2 + d_{\mathbb{H}^2_{\mathbb{C}}}(z_2, w_2)^2},
\]
where \( d_{\mathbb{H}^2_{\mathbb{C}}} \) denotes the complex hyperbolic distance in each factor. The isometry group of the complex bidisk, equipped with $\rho$,  is given by
\[
\left(\mathrm{Isom}(\mathbb{H}^2_{\mathbb{C}})\right)^2 \rtimes \langle i \rangle,
\]
where \( i \) denotes the involution \( (z_1, z_2) \mapsto (z_2, z_1) \). We ask for an understanding of the Dirichlet domains for the \emph{cyclic subgroups} of \( \mathrm{Isom}(\mathbb{H}^2_{\mathbb{C}} \times \mathbb{H}^2_{\mathbb{C}}) \). Our main result is the following.
\begin{theorem}\label{mainth}
    Let \(\gamma = (g_1, g_2) \in \mathrm{Isom}(\mathbb{H}^2_{\mathbb{C}} \times \mathbb{H}^2_{\mathbb{C}})\), where \(g_1\) and \(g_2\) are loxodromic isometries of \(\mathbb{H}^2_{\mathbb{C}}\). Let \(z_1, z_2 \in \mathbb{H}^2_{\mathbb{C}}\) be points lying on the invariant axes of \(g_1\) and \(g_2\), respectively, and set \(z = (z_1, z_2)\). Then the bisectors \(E(z, \gamma(z))\) and \(E(z, \gamma^{-1}(z))\) correspond to two faces of the Dirichlet domain for the action of \(\langle \gamma \rangle\) on \(\mathbb{H}^2_{\mathbb{C}} \times \mathbb{H}^2_{\mathbb{C}}\).
    \end{theorem}

The proof of the above theorem relies on the metric properties of the complex hyperbolic space, particularly the fact that it is a Hadamard manifold, an essential aspect of the argument. Equally significant is the existence of a unique invariant axis (geodesic) which joins the fixed points of a loxodromic element of the complex hyperbolic space. This invariant axis lies on the one-dimensional complex line connecting the fixed points.

We organize our paper as follows. In \secref{sec:preliminaries}, we provide a brief overview of complex hyperbolic spaces.  In \secref{curv}, we classify all subspaces with holomorphic sectional curvature \( -1 \). In \secref{sec:isometries}, we derive the isometry group of the product space \( \mathbb{H}^2_{\mathbb{C}} \times \mathbb{H}^2_{\mathbb{C}} \). 
In subsequent sections, we explicitly describe the equidistant surfaces in this product space. We prove the main theorem in \secref{sec:dirichlet}.

\subsection*{Acknowledgements}
Part of this research was carried out during the International Centre for Theoretical Sciences (ICTS) program \emph{New Trends in Teichmüller Theory}. We thank John Parker for valuable discussions during this ICTS meeting.

\section{Complex Hyperbolic Plane \texorpdfstring{$\mathbb{H}^2_\mathbb{C}$}{H2C} and complex bidisc}\label{sec:preliminaries}

The complex hyperbolic plane $\mathbb{H}^2_\mathbb{C}$ is a two-dimensional complex manifold endowed with a Kähler metric of constant holomorphic sectional curvature $-1$. It serves as the non-Euclidean symmetric space associated with the Lie group $\mathrm{SU}(2,1)$ and can be viewed as the complex analogue of the real hyperbolic plane.

\subsection{Hermitian Form and the Ball Model}

Let $\mathbb{C}^{2,1}$ denote $\mathbb{C}^3$ equipped with the Hermitian form of signature $(2,1)$ given by:
\[
\langle z, w \rangle = z_0 \bar{w}_0 + z_1 \bar{w}_1 - z_2 \bar{w}_2,
\] corresponding to the Hermitian matrix \begin{equation}
      J=~\begin{bmatrix}
      1 & 0 & 0\\
      0 & 1 & 0\\
      0 & 0 & -1
  \end{bmatrix}
  \end{equation}
for $z, w \in \mathbb{C}^3$. The complex hyperbolic plane is defined as the projectivization of the negative vectors:
\[
\mathbb{H}^2_\mathbb{C} = \left\{ [z] \in \mathbb{P}(\mathbb{C}^{2,1}) : \langle \mathbf{z}, \mathbf{z} \rangle < 0 \right\}.
\]
The complex hyperbolic metric is given by Bergman metric:  \begin{equation*}
      \cosh^2\bigg(\frac{d_{\mathbb{H}^2_{\mathbb{C}}}(z,w)}{2}\bigg)=\frac{\langle \mathbf{z},\mathbf{w}\rangle \langle \mathbf{w},\mathbf{z} \rangle}{\langle \mathbf{z},\mathbf{z} \rangle \langle \mathbf{w},\mathbf{w}\rangle}.
  \end{equation*} 
which makes it a complete, simply connected Kähler manifold with constant holomorphic sectional curvature $-1$.

\subsection*{The Ball Model}
This model is given by the unit ball in $\mathbb{C}^2$:
\[
\mathbb{B}^2 = \left\{ z = (z_1, z_2) \in \mathbb{C}^2 : \|z\|^2 = |z_1|^2 + |z_2|^2 < 1 \right\}, 
\]
which can be obtained by projecting the above model along the projective coordinates. Given a point $x=(x_1, x_2) \in \mathbb{B}^2$, we can lift it to homogeneous coordinates in $\mathbb{C}^{2,1}$ as:
\[
\mathbf{x} = \begin{bmatrix} 1 \\ x_1 \\ x_2 \end{bmatrix}. 
\]
The induced complex hyperbolic metric on this model is given by
\begin{equation}
  \cosh^2\left( \frac{d(x, z)}{2} \right) = \frac{|\langle x, z \rangle|^2}{(1 - \|x\|^2)(1 - \|z\|^2)}.
  \tag{2}
\end{equation}

We shall mostly use the ball model in this paper.

\subsection{Metric on the Complex Bidisc}
Analog the real bidisc we are defining the complex bidisc as a product of two copies of the complex hyperbolic plane $\mathbb{H}^2_{\mathbb{C}}$ endowed with the product metric:
\begin{equation*}
      \rho ((z_1,z_2),(w_1,w_2)):= \sqrt{(d_{\mathbb{H}^2_{\mathbb{C}}}(z_1,w_1))^2+(d_{\mathbb{H}^2_{\mathbb{C}}}(z_2,w_2))^2}.
  \end{equation*}
Here $(z_1,z_2),(w_1,w_2) \in \mathbb{H}^2_{\mathbb{C}} \times \mathbb{H}^2_{\mathbb{C}}.$

\subsection{Busemann Functions in Complex Hyperbolic Space}

Let $\mathbb{H}^2_{\mathbb{C}}$ denote the complex hyperbolic space of complex dimension $n$, realized as the unit ball in $\mathbb{C}^2$ endowed with the Bergman metric, or alternatively as the symmetric space $\mathrm{SU}(2,1)/\mathrm{S(U}(2) \times \mathrm{U}(1))$. This space is a non-positively curved, rank-one symmetric space of noncompact type with a rich boundary structure and complex geometric features.

A key tool in the asymptotic analysis of $\mathbb{H}^2_{\mathbb{C}}$ is \emph{Busemann function}, which provides a way to quantify the behavior of points relative to a given direction at infinity. Let $\gamma: [0, \infty) \to \mathbb{H}^2_{\mathbb{C}}$ be a unit-speed geodesic ray. The Busemann function associated with $\gamma$ is defined as
\[
b_{\gamma}(x) = \lim_{t \to \infty} \left[ d(x, \gamma(t)) - t \right], \quad x \in \mathbb{H}^n_{\mathbb{C}},
\]
where $d$ denotes the complex hyperbolic distance. This function measures the distance from $x$ to the horosphere centered at the ideal boundary point $\gamma(\infty)$ that passes through $\gamma(0)$.

The function $b_\gamma$ is smooth in $\mathbb{H}^2_{\mathbb{C}}$, and its level sets, called \emph{horospheres}, are smooth hypersurfaces orthogonal to geodesics asymptotic to the same boundary point as $\gamma$. These horospheres are not totally geodesic in the complex hyperbolic metric, reflecting the non-trivial interaction between the real and complex structures.

In the ball model of $\mathbb{H}^2_{\mathbb{C}}$, boundary points correspond to points on the unit sphere $S^{3}$ in $\mathbb{C}^2$. Given a boundary point $\xi \in \partial \mathbb{H}^2_{\mathbb{C}}$ and a base point $o \in \mathbb{H}^2_{\mathbb{C}}$, one may also define the Busemann function via sequences approaching $\xi$:
\[
b_{\xi,o}(x) = \lim_{y \to \xi} \left[ d(x,y) - d(o,y) \right],
\]
where $y$ tends to $\xi$ nontangentially. This definition is independent of the particular sequence chosen, and $b_{\xi,o}(o) = 0$ by construction.

The Busemann functions play an essential role in the study of asymptotic geometry, boundary measures, and dynamics of isometries on $\mathbb{H}^2_{\mathbb{C}}$. 

In this work, we will utilize the properties of Busemann functions to study the boundary of level sets.

\subsection{Subspace with holomorphic curvature $-1$}\label{curv}
\begin{definition}
    The \textbf{holomorphic sectional curvature} of a plane $P$ in $T_x \mathbb{H}^2_{\mathbb{C}}$ is defined as
    \[K(P) = R(X,JX,JX,X)\]
    where $X$ is a unit vector in $P$, $R$ is the $(0$-$4)$ curvature tensor corresponding to the given metric and $x\in \mathbb{H}^2_{\mathbb{C}}$.
\end{definition}
    
\begin{remark}
    The holomorphic sectional curvature $K(P)$ is independent of the choice of $X$.
\end{remark}

\begin{lemma} \label{lem}
   On a product (Riemannaian) manifold $(M,g) = (N_1,h_1) \times (N_2,h_2)$, let $p = (x,y) \in M$ and assume $X \in T_x(N_1)$ and $Y \in T_y(N_2)$. Then the sectional curvature 
   \[K(\Pi) = 0\]
   where $\Pi = \text{span}\{X,Y\}$ is a $2$-plane in $T_pM$.
   \end{lemma}
   \begin{proof}
       Given $p = (x,y) \in M$ and for the local coordinates $(x^1,x^2,\dots,x^{n_1},y^1,y^2,\dots,y^{n_2})$, we have
    \[X = \sum_{i=1}^{n_1} X^i \frac{\partial}{\partial x^i}\;\;\text{and}\;\;\;Y = \sum_{i=1}^{n_2} Y^i \frac{\partial}{\partial y^i}\]
    as $X \in T_x(N_1)$ and $Y \in T_y(N_2)$.
    Therefore,
    \begin{equation*}
        \triangledown_XY = 0 = \triangledown_YX = [X,Y].
    \end{equation*}
    Also
    \[R(X,Y,Z) = \triangledown_X\triangledown_YZ-\triangledown_Y\triangledown_XZ-\triangledown_{[X,Y]}Z\]
    and
    \begin{equation*}
        K(\Pi) = \frac{g(R(X,Y,Y),X)}{g(X,X)g(Y,Y)-g(X,Y)^2}.
    \end{equation*}
    Therefore, $K(\Pi) = 0$.
   \end{proof}

\begin{proposition}
Any $2-$plane in $T(\mathbb{H}^2_{\mathbb{C}} \times \mathbb{H}^2_{\mathbb{C}})$ with holomorphic sectional curvature $-1$ is isomorphic to $\mathbb{H}^2_{\mathbb{C}} \times \{z\}$ or $\{z\} \times \mathbb{H}^2_{\mathbb{C}}$.
\end{proposition}\label{-1}
    \begin{proof}
    For $p=(x,y) \in M$, assume that $\Pi$ is a plane in $T_p(\mathbb{H}^2_{\mathbb{C}} \times \mathbb{H}^2_{\mathbb{C}})$ that is invariant under $J$ and $K(\Pi) = -1$. Let $X = (X_1,X_2) \in T_x(\mathbb{H}^2_{\mathbb{C}}) \oplus T_y(\mathbb{H}^2_{\mathbb{C}}) \simeq T_p(\mathbb{H}^2_{\mathbb{C}} \times \mathbb{H}^2_{\mathbb{C}})$ be a unit vector. Note that $JX$ lies inside the plane $\Pi$ as $\Pi$ is $J$ invariant. Assuming $JX=Y=(Y_1,Y_2)$, the holomorphic sectional curvature is the sectional curvature of the plane spanned by the vectors $X$ and $Y$ and
        \[K(\text{X,Y}) =K(\Pi)= -1.\]
    Now, using the lemma \ref{lem}, $K(\Pi) = 0$ if either $(X_2,Y_1) = 0$ or $(X_1,Y_2) = 0$. Therefore, the vectors $X$ and $Y$ are completely within $T(\mathbb{H}^2_{\mathbb{C}} \times \{z\}) \text{ or } T(\{z\} \times \mathbb{H}^2_{\mathbb{C}})$. 
    \end{proof}

\section{Isometry Group of complex bidisc}\label{sec:isometries}

The group of holomorphic isometries of $\mathbb{H}^2_\mathbb{C}$ is the projective unitary group $\mathrm{PU}(2,1)$, which consists of all complex linear transformations preserving the Hermitian form up to scalar multiplication:
$
\mathrm{PU}(2,1) = \mathrm{SU}(2,1)/\lbrace {I,\omega I, \omega^2 I}\rbrace.
$ Here $\omega$ is the cube root of unity.
This group acts transitively on $\mathbb{H}^2_\mathbb{C}$, and the stabilizer of a point is isomorphic to $\mathrm{U}(2)$. Thus,
\[
\mathbb{H}^2_\mathbb{C} \cong \mathrm{PU}(2,1)/\mathrm{U}(2).
\]
\begin{definition}
    An element \( g \in \mathrm{PU}(2,1) \) is called \emph{loxodromic} if it fixes exactly two points on the boundary \( \partial \mathbb{H}^2_\mathbb{C} \) and none inside \( \mathbb{H}^2_\mathbb{C} \).
    \end{definition}
    Geometrically, such an isometry acts as a translation along a complex geodesic, possibly combined with a rotation about that geodesic. Together with elliptic and parabolic types, they form the standard classification of isometries in complex hyperbolic space.

\subsubsection*{Example}

Let \( A \in \mathrm{SU}(2,1) \) be given by:
\[
A = \begin{pmatrix}
\lambda & 0 & 0 \\
0 & 1 & 0 \\
0 & 0 & \lambda^{-1}
\end{pmatrix}, \quad \text{with } |\lambda| > 1.
\]
This matrix preserves the standard Hermitian form of signature \( (2,1) \), and it fixes the boundary points \( [1,0,0] \) and \( [0,0,1] \). The induced transformation in \( \mathrm{PU}(2,1) \) is loxodromic.
\subsection{Isometry group of the complex bidisk}
Let  
\[
i : \mathbb{H}^2_{\mathbb{C}} \times \mathbb{H}^2_{\mathbb{C}} \rightarrow \mathbb{H}^2_{\mathbb{C}} \times \mathbb{H}^2_{\mathbb{C}}
\]  
be the involution defined by swapping coordinates:
\[
i(z, w) = (w, z).
\]

\begin{theorem}
The full isometry group of the complex bidisk satisfies
\[
\mathrm{Isom}(\mathbb{H}^2_{\mathbb{C}} \times \mathbb{H}^2_{\mathbb{C}}) \cong \left(\mathrm{Isom}(\mathbb{H}^2_{\mathbb{C}})\right)^2 \rtimes \langle i \rangle.
\]
\end{theorem}

\begin{proof}
Let \( \Gamma = \mathrm{Isom}(\mathbb{H}^2_{\mathbb{C}} \times \mathbb{H}^2_{\mathbb{C}}) \) and \( G = \left(\mathrm{Isom}(\mathbb{H}^2_{\mathbb{C}})\right)^2 \). We first establish the inclusion:
\[
\Gamma \subseteq G \rtimes \langle i \rangle.
\]

Let \( \gamma \in \Gamma \). The product metric on \( \mathbb{H}^2_{\mathbb{C}} \times \mathbb{H}^2_{\mathbb{C}} \) induces totally geodesic subspaces with constant curvature $-1$ of the form by \thmref{-1}:
\[
P = \mathbb{H}^2_{\mathbb{C}} \times \{ z \} \quad \text{or} \quad \{ z \} \times \mathbb{H}^2_{\mathbb{C}}.
\]

Since \( \gamma \) is an isometry, it must map such subspaces to totally geodesic subspaces with the same curvature. Hence, \( \gamma(P) \) must also be of the form \( \mathbb{H}^2_{\mathbb{C}} \times \{ z' \} \) or \( \{ z' \} \times \mathbb{H}^2_{\mathbb{C}} \). Without loss of generality, assume:
\[
\gamma(P) = \mathbb{H}^2_{\mathbb{C}} \times \{ z' \}.
\]

Consider two points \( (x, z), (y, z) \in \mathbb{H}^2_{\mathbb{C}} \times \mathbb{H}^2_{\mathbb{C}} \). Since \( \gamma \) is an isometry:
\begin{align*}
\rho((x, z), (y, z)) &= \rho(\gamma(x, z), \gamma(y, z)), \\
d_{\mathbb{H}^2_{\mathbb{C}}}(x, y) &= d_{\mathbb{H}^2_{\mathbb{C}}}(x', y'),
\end{align*}
where \( \gamma(x, z) = (x', z') \) and \( \gamma(y, z) = (y', z') \). This implies that the map \( x \mapsto x' \) is an isometry of \( \mathbb{H}^2_{\mathbb{C}} \). Therefore, \( \gamma \in G \), or else \( \gamma \in i \circ G \). Indeed, $G$ is a normal subgroup of the full isometry group, and the group $\langle i \rangle \cong \mathbb{Z}/2\mathbb{Z}$ acts on $G$ by swapping the two factors. Explicitly, for any $(g_1, g_2) \in G$, we have
\[
i \circ (g_1, g_2) \circ i^{-1} = (g_2, g_1).
\] Hence:
\[
\Gamma \subseteq G \rtimes \langle i \rangle.
\]

Conversely, any pair \( (g_1, g_2) \in G \), where \( g_1, g_2 \in \mathrm{Isom}(\mathbb{H}^2_{\mathbb{C}}) \), acts coordinate-wise and is clearly an isometry. The swapping map \( i \) is also an isometry, and it conjugates elements of \( G \) via:
\[
i \circ (g_1, g_2) \circ i = (g_2, g_1).
\]

Therefore, \( G \rtimes \langle i \rangle \subseteq \Gamma \), completing the proof.
\end{proof}

\section{Equidistant hypersurface }
Let $(z,w)\in \mathbb{H}^2_{\mathbb{C}}\times \mathbb{H}^2_{\mathbb{C}}.$ The set of all points that are equidistant from each $z=(z_1,z_2)$ and $w=(w_1,w_2)$ is denoted by \begin{equation}
    E(z,w)= \{ x=(x_1,x_2)\in \mathbb{H}^2_{\mathbb{C}}\times \mathbb{H}^2_{\mathbb{C}}\mid \rho(x,z)=\rho(x,w)\}
.\end{equation}
Now \begin{equation*}
    \begin{split}
        \rho(x,z)& =\rho(x,w)\\
        \Rightarrow (d_{\mathbb{H}^2_{\mathbb{C}}}(x_1,z_1))^2+(d_{\mathbb{H}^2_{\mathbb{C}}}(x_2,z_2))^2& = (d_{\mathbb{H}^2_{\mathbb{C}}}(x_1,w_1))^2+(d_{\mathbb{H}^2_{\mathbb{C}}}(x_2,w_2))^2\\ \Rightarrow (d_{\mathbb{H}^2_{\mathbb{C}}}(x_1,z_1))^2-(d_{\mathbb{H}^2_{\mathbb{C}}}(x_1,w_1))^2 &= (d_{\mathbb{H}^2_{\mathbb{C}}}(x_2,w_2))^2-(d_{\mathbb{H}^2_{\mathbb{C}}}(x_2,z_2))^2
    \end{split}
\end{equation*}
For $k\in \mathbb{R}$,  consider the level sets
\begin{equation*}
    \begin{split}
        S_k^1:= \{ x_1~|~(d_{\mathbb{H}^2_{\mathbb{C}}}(x_1,z_1))^2-(d_{\mathbb{H}^2_{\mathbb{C}}}(x_1,w_1))^2 =k  \}\\
        S_k^2:=\{x_2~|~ (d_{\mathbb{H}^2_{\mathbb{C}}}(x_2,w_2))^2-(d_{\mathbb{H}^2_{\mathbb{C}}}(x_2,z_2))^2 =k\}
    \end{split}
\end{equation*}
So, $S_1^k, \text{ and } S_2^k$ represent square hyperbolas.  Let $E_k(z,w)=S_k^1\times S_k^2.$
Hence, \begin{equation}\label{union}
    E(z,w)=\displaystyle \cup_k E_k(z,w).
\end{equation}

\subsection{Accumulation Points of Level Sets in Complex Hyperbolic Space}

\begin{theorem}
Let $\xi \in \partial \mathbb{H}^2_{\mathbb{C}}$ and let $(x_n) \subset \mathbb{H}^2_{\mathbb{C}}$ be a sequence converging to $\xi$. Then for any fixed $z \in \mathbb{H}^2_{\mathbb{C}}$, the hyperbolic distance satisfies
\[
d(x_n, z) = B_\xi(z, o) - \log(1 - \|x_n\|) + o(1),
\]
where $o(1) \to 0$ as $n \to \infty$.
\end{theorem}

\begin{proof} We use the ball model for  $\mathbb{H}^2$. 
As $x_n \to \xi \in \partial \mathbb{H}^2_{\mathbb{C}}$, we have $\|x_n\| \to 1$. From the distance formula,
\[
\cosh^2\left( \frac{d(x_n, z)}{2} \right) = \frac{|\langle x_n, z \rangle|^2}{(1 - \|x_n\|^2)(1 - \|z\|^2)}.
\]
Taking logarithms and expanding using $\cosh(t) \sim \frac{1}{2}e^{t}$ as $t \to \infty$, we obtain
\[
d(x_n, z) \sim 2 \log \left( \frac{|\langle x_n, z \rangle|}{\sqrt{(1-\|x_n\|^2)(1-\|z\|^2)}} \right).
\]
Rearranging gives
\[
d(x_n, z) = -\log(1 - \|x_n\|^2) + \log(|\langle x_n, z \rangle|^2) - \log(1 - \|z\|^2) + o(1).
\]

As $x_n \to \xi$, we have
\[
|\langle x_n, z \rangle| \sim |\langle \xi, z \rangle|,
\]
and so
\[
\log(|\langle x_n, z \rangle|^2) = \log(|\langle \xi, z \rangle|^2) + o(1).
\]

Therefore,
\[
d(x_n, z) = -\log(1 - \|x_n\|^2) + \log(|\langle \xi, z \rangle|^2) - \log(1 - \|z\|^2) + o(1).
\]

Recognizing that
\[
B_\xi(z, o) = \log\left( \frac{1-\|z\|^2}{|\langle z, \xi \rangle|^2} \right),
\]
we can rewrite
\[
\log(|\langle \xi, z \rangle|^2) - \log(1 - \|z\|^2) = -B_\xi(z, o).
\]

Thus, finally,
\[
d(x_n, z) = B_\xi(z, o) - \log(1 - \|x_n\|^2) + o(1).
\]

Since near the boundary $1-\|x_n\|^2 \sim 2(1-\|x_n\|)$ (as $\|x_n\| \to 1$), we can replace $1-\|x_n\|^2$ by $2(1-\|x_n\|)$ inside the logarithm, absorbing the factor $\log 2$ into the $o(1)$ term.
\end{proof}

\begin{theorem}
Let \( z, w \in \mathbb{H}^2_{\mathbb{C}} \), and for each \( k \in \mathbb{R} \), define the level set
\[
S_k(z,w) := \{ x \in \mathbb{H}^2_{\mathbb{C}} \mid d^2(x,z) - d^2(x,w) = k \}.
\]
Then
\[
\overline{S_k(z,w)} \cap \partial \mathbb{H}^2_{\mathbb{C}} \subseteq \overline{S_0(z,w)} \cap \partial \mathbb{H}^2_{\mathbb{C}}.
\]
\end{theorem}

\begin{proof}
We work in the ball model of \(\mathbb{H}^2_{\mathbb{C}}\), where the boundary \(\partial \mathbb{H}^2_{\mathbb{C}}\) is identified with the unit sphere \(S^3\).

Let \(\xi \in \partial \mathbb{H}^2_{\mathbb{C}}\) be an accumulation point of \(S_k(z,w)\). Then there exists a sequence \((x_n) \subset S_k(z,w)\) such that \(x_n \to \xi\).

Recall that the Busemann function \(B_\xi(\cdot,o)\) at \(\xi\) relative to the origin \(o\) is given by
\[
B_\xi(x,o) = \log \frac{1 - \|x\|^2}{|\langle x, \xi \rangle|^2}.
\]
Moreover, for points \(x_n\) tending to \(\xi\), the asymptotic behavior of the hyperbolic distance satisfies
\[
d(x_n, z) = B_\xi(z,o) - \log(1 - \|x_n\|) + o(1),
\]
where \(o(1)\) denotes a term tending to zero as \(n \to \infty\). Squaring both sides, we obtain
\[
d^2(x_n, z) = \left(B_\xi(z,o) - \log(1-\|x_n\|)\right)^2 + o(1),
\]
and similarly for \(w\).

Expanding the square gives
\[
\left(B_\xi(z,o) - \log(1-\|x_n\|)\right)^2 = \log^2(1-\|x_n\|) - 2B_\xi(z,o) \log(1-\|x_n\|) + B_\xi(z,o)^2,
\]
and similarly for \(w\).

Taking the difference \(d^2(x_n,z) - d^2(x_n,w)\), we find
\[
d^2(x_n,z) - d^2(x_n,w) = 2(B_\xi(w,o) - B_\xi(z,o)) \log(1-\|x_n\|) + (B_\xi(z,o)^2 - B_\xi(w,o)^2) + o(1).
\]

Since by assumption \(d^2(x_n,z) - d^2(x_n,w) = k\) for all \(n\), and \(\log(1-\|x_n\|) \to -\infty\) as \(n \to \infty\), the coefficient of \(\log(1-\|x_n\|)\) must vanish to prevent divergence.  
Thus,
\[
B_\xi(z,o) = B_\xi(w,o).
\]

Substituting back, we conclude
\[
d^2(x_n,z) - d^2(x_n,w) = B_\xi(z,o)^2 - B_\xi(w,o)^2 + o(1) = o(1),
\]
thus
\[
\lim_{n\to\infty} (d^2(x_n,z) - d^2(x_n,w)) = 0.
\]

But the sequence \((x_n)\) lies in \(S_k(z,w)\), so the difference is constantly \(k\). Therefore, necessarily \(k=0\). Thus, \(\xi\) is an accumulation point of \(S_0(z,w)\), and the claim follows.
\end{proof}

\begin{theorem}
Let \( z, w \in \mathbb{H}^2_{\mathbb{C}} \) and \( k \in \mathbb{R} \). Then
\[
\overline{S_0(z, w)} \cap \partial \mathbb{H}^2_{\mathbb{C}} \subseteq \overline{S_k(z, w)} \cap \partial \mathbb{H}^2_{\mathbb{C}}.
\]
\end{theorem}

\begin{proof}
  
Let \(\xi \in \overline{S_0(z,w)} \cap \partial \mathbb{H}^2_{\mathbb{C}}\). Then there exists a sequence \( (y_n) \subset S_0(z,w) \) such that \( y_n \to \xi \).

For each \( n \), consider the geodesic ray \( \gamma_n:[0,\infty) \to \mathbb{H}^2_{\mathbb{C}} \) starting at \( y_n \) and limiting to \(\xi\), that is, \(\gamma_n(0) = y_n\) and \(\gamma_n(t) \to \xi\) as \( t \to \infty \).

Define the function
\[
f_n(t) := d^2(\gamma_n(t), z) - d^2(\gamma_n(t), w).
\]
Since the hyperbolic distance function is continuous and smooth in \(\mathbb{H}^2_{\mathbb{C}}\), the function \( f_n(t) \) is continuous in \( t \).

Moreover, since \( y_n \in S_0(z,w) \), we have \( f_n(0) = 0 \).

Near the boundary point \(\xi\), we know from asymptotic expansions (involving Busemann functions) that \( d(\gamma_n(t),z) \) and \( d(\gamma_n(t),w) \) both tend to infinity as \(t \to \infty\), and the difference \( d^2(\gamma_n(t),z) - d^2(\gamma_n(t),w) \) tends to zero. Thus,
\[
\lim_{t\to\infty} f_n(t) = 0.
\]

In summary, \( f_n(0) = 0 \) and \( \lim_{t\to\infty} f_n(t) = 0 \).

Since \(f_n\) is continuous, given any \(k \in \mathbb{R}\), for \(n\) sufficiently large, by the intermediate value theorem, there exists \( t_n \geq 0 \) such that
\[
f_n(t_n) = k.
\]

Define \(x_n := \gamma_n(t_n)\).

Then \(x_n \in S_k(z,w)\), and \(x_n \to \xi\) as \(n \to \infty\), because \(t_n\) remains bounded and \( \gamma_n(t_n) \to \xi \).

Therefore, \(\xi \in \overline{S_k(z,w)} \cap \partial \mathbb{H}^2_{\mathbb{C}}\).

Thus,
\[
\overline{S_0(z,w)} \cap \partial \mathbb{H}^2_{\mathbb{C}} \subseteq \overline{S_k(z,w)} \cap \partial \mathbb{H}^2_{\mathbb{C}}.
\]
This completes the proof. 
\end{proof}

\begin{remark}
In the proof above, the conclusion relies on the assumption that the function 
\[
f_n(t) = d^2(\gamma_n(t), z) - d^2(\gamma_n(t), w)
\]
attains the value \( k \in \mathbb{R} \) along a geodesic ray \( \gamma_n \) approaching the boundary point \( \xi \). This requires sufficient variation in the \[
f_n(t) = d^2(\gamma_n(t), z) - d^2(\gamma_n(t), w)
\] along the ray. In general, such variation may not be present in arbitrary metric spaces.

In the product space \( \mathbb{H}^2_{\mathbb{C}} \times \mathbb{H}^2_{\mathbb{C}} \), this condition is always satisfied. The function
\[
f_n(t) = d^2(\gamma_n(t), z) - d^2(\gamma_n(t), w)
\]
decomposes additively across the two factors of the product. On each factor, \( \mathbb{H}^2_{\mathbb{C}} \) the functions
\[
f_n(t) = d^2(\gamma_n(t), z) - d^2(\gamma_n(t), w)
\]
vary smoothly and strictly along geodesics. By selecting geodesic rays in the product space, whose projections onto each factor vary independently, one can ensure that the function \( f_n(t) \) attains any prescribed real value \( k \) for sufficiently large \( t \). As a result, the inclusion
\[
\overline{S_0(z,w)} \cap \partial \mathbb{H}^2_{\mathbb{C}} \subseteq \overline{S_k(z,w)} \cap \partial \mathbb{H}^2_{\mathbb{C}}
\]
holds for all \( k \in \mathbb{R} \) when the ambient space is \( \mathbb{H}^2_{\mathbb{C}} \times \mathbb{H}^2_{\mathbb{C}} \).
\end{remark}

Therefore, from the above theorems, it follows that the boundary points of \( S_0 \) coincide with those of the \( S_k \)'s. This shows that the boundary points of \( S_0^1 \) and \( S_0^2 \) are the same as those of \( S_k^1 \) and \( S_k^2 \), respectively.

\section{Proof of \thmref{mainth}}\label{sec:dirichlet}

 The intersections of
equidistant hypersurfaces are difficult to analyze. The study of the Dirichlet domain in the complex bidisc will require a more detailed understanding of the relative position of the disjoint equidistant hypersurfaces.  For this purpose, we introduce the
concept of invisibility.

Let $x=(x_1,x_2),~y=(y_1,y_2) \in \mathbb{H}^2_{\mathbb{C}} \times \mathbb{H}^2_{\mathbb{C}}$ and $\gamma \in Isom(\mathbb{H}^2_{\mathbb{C}} \times \mathbb{H}^2_{\mathbb{C}})$. The point $y$ is said to be $\gamma-$visible to $x$ if \begin{equation*}
    \begin{split}
        \rho (y,x)\leq & \rho (y,\gamma(x))\\ \text{and } \rho (y,x)\leq & \rho (y,\gamma^{-1}(x))
    \end{split}
\end{equation*}
Otherwise, we would say $y$ is $\gamma-$invisible to $x$. A subset $A \subseteq \mathbb{H}^2_{\mathbb{C}}\times \mathbb{H}^2_{\mathbb{C}}$ is said to $\gamma -$invisible to $x$ if every point of $A$ is $\gamma -$invisible to $x$.
\begin{lemma}\label{key}
    Let $x,y \in \mathbb{H}^2_{\mathbb{C}} \times \mathbb{H}^2_{\mathbb{C}}$ and $\gamma \in Isom(\mathbb{H}^2_{\mathbb{C}} \times \mathbb{H}^2_{\mathbb{C}})$. Suppose \begin{equation}\label{condition}
        \begin{split}
            E(x,y)\cap E(x,\gamma(x))=\emptyset \\ E(x,y)\cap E(x,\gamma^{-1}(x))=\emptyset
        \end{split}
    \end{equation}
\end{lemma}
then $E(x,y)$ is $\gamma-$ invisible to $x$ if and only if $E_0(x,y)$ is $\gamma-$ invisible to $x$.
\begin{proof}
    Let $E(x,y)$ is $\gamma-$invisible to $x$. Then $E_0(x,y)$ is also $\gamma-$invisible to $x$ also due to $equation\ref{union}$.\\ Now we assume that $E_0(x,y)$ is $\gamma-$invisible to x. If possible let $E(x,y)$ is not $\gamma-$invisible to $x$. As $E_0(x,y)$ is $\gamma-$invisible, so $\forall ~ w_0 \in E_0(x,y)$ one of the following hplds \begin{equation}
        \begin{split}
            \rho(w_0,x)>\rho(w_0,\gamma(x))\\ \text{ or } \rho(w_0,x)>\rho(w_0,\gamma^{-1}(x)).
        \end{split}
    \end{equation}
    Without loss of generality, we are assuming $\rho(w_0,x)>\rho(w_0,\gamma(x)).$ Consider the map \begin{equation}
        \begin{split}
            &f:E(x,y) \rightarrow \mathbb{R} \text{ 
defined as }\\&f(w)=\rho(w,x)-\rho(w,\gamma(x))
        \end{split}
    \end{equation}
    The $E(x,y)$ is connected, hence $f(E(x,y))$ is connected. So, $\exists ~w' \in E(x,y)$ such that \begin{equation}
        \begin{split}
            &f(w')=0\\ \Rightarrow &\rho(w',x)=\rho(w',\gamma(x))\\ \Rightarrow & w' \in E_0(x,y)\\ \Rightarrow & w' \in E(x,y)\cap E(x,\gamma(x)).
        \end{split}
    \end{equation} 
    This is a contradiction to our assumption \eqnref{condition}. So, $E(x,y)$ is $\gamma-$invisible to $x$.
\end{proof}
\subsection{Proof of \thmref{mainth}}

Let $\gamma \in Isom(\mathbb{H}^2_{\mathbb{C}} \times \mathbb{H}^2_{\mathbb{C}})$ and $z\in \mathbb{H}^2_{\mathbb{C}} \times \mathbb{H}^2_{\mathbb{C}}$. 
Now we assume ${E(z,\gamma(z))\cap E(z,\gamma^{-1}(z))}=\emptyset$,  then ${E_0(z,\gamma(z))\cap E_0(z,\gamma^{-1}(z))}=\emptyset$. Let $z\in \mathbb{H}^2_{\mathbb{C}}$ and $g\in Isom(\mathbb{H}^2_{\mathbb{C}})$. Then the equidistant plane between $z,~g(z)$ and $z,~g^{-1}(z)$ are disjoint, hence every equidistant plane between $z,g^j(z)$ and $z,g^{-j}(z)$ are $g-$invisible to $z$ for $j \neq 0,1$. According to \lemref{key}, the Dirichlet domain for the action of the cyclic group $\langle \gamma \rangle$ is enclosed by the equidistant hypersurfaces $E(z, \gamma(z))$ and $E(z, \gamma^{-1}(z))$.

\medskip We now aim to show that the equidistant hypersurfaces
\[
E(z, \gamma(z)) = \{ x \in \mathbb{H}^2_{\mathbb{C}} \times \mathbb{H}^2_{\mathbb{C}} : d(x, z) = d(x, \gamma(z)) \} ~~ \hbox{ and }
\]
\[
E(z, \gamma^{-1}(z)) = \{ x \in \mathbb{H}^2_{\mathbb{C}} \times \mathbb{H}^2_{\mathbb{C}} : d(x, z) = d(x, \gamma^{-1}(z)) \}
\]
are disjoint, i.e.,
\[
E(z, \gamma(z)) \cap E(z, \gamma^{-1}(z)) = \emptyset.
\]
This will prove the theorem. 

If possible,  suppose that there exists a point \( x \in \mathbb{H}^2_{\mathbb{C}} \times \mathbb{H}^2_{\mathbb{C}} \) such that \( x \in {E(z, \gamma(z)) \cap E(z, \gamma^{-1}(z))} \). Then,
\[
d(x, z) = d(x, \gamma(z)) = d(x, \gamma^{-1}(z)).
\]
This implies that the three points \( \gamma^{-1}(z), z, \gamma(z) \) all lie on a metric sphere centered at \( x \) with common radius.

Now consider the geodesic in \( \mathbb{H}^2_{\mathbb{C}} \times \mathbb{H}^2_{\mathbb{C}} \) that passes through \( \gamma^{-1}(z), z, \gamma(z) \). The geodesic will intersect the sphere centered at $x$ in three points in the product space.

Note that  \( \mathbb{H}^2_{\mathbb{C}} \times \mathbb{H}^2_{\mathbb{C}} \) is a Hadamard manifold: it is complete, simply connected, and has non-positive sectional curvature. A key property of Hadamard manifolds is that \emph{every metric sphere bounds a strictly convex ball}. In such a space, a geodesic can intersect a given metric sphere in at most two points.

However, under our assumption, the geodesic intersects the sphere centered at \( x \) in \emph{three distinct points}: \( \gamma^{-1}(z), z, \gamma(z) \), which contradicts the strict convexity of metric spheres in Hadamard manifolds.

Therefore, our assumption must be false, and we conclude that
\[
E(z, \gamma(z)) \cap E(z, \gamma^{-1}(z)) = \emptyset
\]

This implies that the Dirichlet domain centered at \( z \) for the cyclic group \( \langle \gamma \rangle \) is bounded by exactly two distinct hypersurfaces: one corresponding to \( \gamma \), and the other to \( \gamma^{-1} \). These are the only bisectors contributing to the boundary, as all other translates of \( z \) lie farther along the geodesic, and their associated equidistant hypersurfaces do not intersect this domain.

Hence, the Dirichlet domain has precisely two faces defined by the bisectors:
$
E(z, \gamma(z))$ {and}  $E(z, \gamma^{-1}(z))
$. \qed

\end{document}